\documentclass[12pt,reqno]{amsart}
\usepackage{amsmath,amssymb,amsfonts,amsthm}
\usepackage{mathrsfs}
\usepackage{soul}
\usepackage[all]{xy}
\usepackage{stmaryrd}
\usepackage{mathtools}
\usepackage{indentfirst}
\usepackage{verbatim}
\usepackage{hyperref}
\usepackage{color}
\usepackage[usenames,dvipsnames]{xcolor}
\usepackage{graphicx}
\graphicspath{ {images/} }
\usepackage[normalem]{ulem}
\usepackage{tikz}
\usepackage{nicefrac}
\usepackage{xfrac}
\usepackage[shortlabels]{enumitem}
\usepackage{comment}

\newtheorem*{claim}{Claim}
\newtheorem*{theorem*}{Theorem}
\newtheorem{theorem}{Theorem}[section]
\newtheorem{proposition}[theorem]{Proposition}
\newtheorem{lemma}[theorem]{Lemma}
\newtheorem{lemmadef}[theorem]{Lemma-Definition}

\theoremstyle{definition}

\newtheorem{definition}[theorem]{Definition}
\newtheorem{remark}[theorem]{Remark}
\newtheorem{example}[theorem]{Example}
\newtheorem{question}[theorem]{Question}

\newcommand{\mc}{\mathcal}

\newcommand{\mb}{\mathbb}

\newcommand{\on}{\operatorname}
\newcommand{\diff}{\textnormal{d}}
\newenvironment{claimproof}[1]{\par\noindent\emph{Proof of Claim.}\space#1}{}
\usepackage[margin=0.75in]{geometry}

\newcommand\cF{{\mathcal{F}}}
\newcommand\cG{{\mathcal{G}}}

\newcommand\cI{{\mathcal{I}}}

\newcommand\cO{{\mathcal{O}}}

\newcommand\cT{{\mathcal{T}}}

\newcommand\bC{{\mathbb C}}

\newcommand\bN{{\mathbb N}}
\newcommand\bP{{\mathbb P}}
\newcommand\bQ{{\mathbb Q}}
\newcommand\bR{{\mathbb R}}

\newcommand\bZ{{\mathbb Z}}

\newcommand{\dlctl}{\delta\text{-}\on{\underline{LCT}}}
\newcommand{\dlctu}{\delta\text{-}\on{\overline{LCT}}}

\title{Boundedness of toric foliations}
\author{Chih-Wei Chang}
\address{Department of Mathematics, National Taiwan University, Taipei, 106, Taiwan}
\email{cwchang0219@ntu.edu.tw}
\author{Yen-An Chen}
\address{Department of Mathematics, National Taiwan University, Taipei, 106, Taiwan}
\email{yachen@ntu.edu.tw}

\begin{document}

\begin{abstract}
    We discuss boundedness of toric Fano foliations and connectedness of its dicritical and singular loci. 
    Moreover, we show the set of interpolated $\delta$-lcts for the toric foliations satisfies the descending chain condition. 
\end{abstract}
\maketitle

\section{Introduction}
In recent years, significant advancements have been made in the minimal model program for foliations, allowing for the extension of results from classical birational geometry to the realm of the foliations, particularly within the framework of the minimal model program. 
Unlike the classical approach, which utilizes the canonical divisor $K_X$ of the ambient variety, the theory of foliations employs the canonical divisor $K_\cF$ associated with the foliation $\cF$ to investigate its geometric properties for both ambient variety and the foliation itself. 
This method constitutes a generalization, as when $\cF=\cT_X$, it reverts the theory of foliations to the classical context. 

Motivated by the Borisov-Alexeev-Borisov Conjecture, which is proved by Birkar, we propose the following question regarding the similar phenomena within the context of foliations:

\begin{question}
Let $\cF$ be a Fano foliation with mild singularities on a normal projective variety $X$. 
Is the set $\{(X,\cF)\}$ bounded?
\end{question}

\cite[Example 3.9]{CJV24} provides a negative answer when both $X$ and $\cF$ have at worst log canonical singularities. 
Also, Fano foliations of rank one on surfaces must contain a dicritical point, which is log canonical or worse. 
Consequently, the condition of having $\varepsilon$-log canonical singularities will render the set empty.

However, inspired by the works of \cite{PS19}, \cite{Che21}, \cite{spicer2023effective}, and \cite{M7}, it is natural to propose the following modified question:

\begin{question}
    Let $(X,\cF,\Delta,t)$ be an adjoint foliated structure with at worst $\delta$-log canonical singularities for $\delta\in (0,1]$. 
    Suppose $-K_{(X,\cF,\Delta,t)}$ is ample. 
    Is the set $\{(X,\cF)\}$ bounded?
\end{question}

We show it holds when $(X,\cF,\Delta,t)$ is toric:
\begin{theorem}[{Theorem~\ref{thm:bounded_adjoint_Fano}}]
    Fix a positive integer $n$, a non-negative integer $r\leq n$, a positive real number $\delta$, and two real numbers $t_i\in [0,1)$ for $i=1$, $2$. 
    Let $\cF=\cF_W$ be a toric foliation of rank $r$ on a complete toric variety $X=X_\Sigma$ of dimension $n$, $\Delta$ be an effective torus-invariant divisor on $X_\Sigma$, and $K_{t_i} = t_iK_\cF+(1-t_i)K_X+\Delta$ for $i=1$, $2$. 
    Suppose $-K_{t_1}$ is ample and the adjoint foliated structure $(X,\cF,\Delta,t_2)$ is $\delta$-lc. 
    Then we have the following:
    \begin{enumerate}
        \item There is an effective divisor $\Delta'$ such that $K_X+\Delta'$ is $\bR$-Cartier. 
        \item $(X,\Delta')$ is $(1-t_2)\delta$-lc. In particular, $X$ is potentially klt. 
        \item The number of isomorphism classes of such $X$ is finite, which depends only on $t$, $\delta$, and $n$. 
        \item The set consisting of $(X,\cF)$ is bounded. 
    \end{enumerate}
\end{theorem}

We then study the dicritical and singular loci for Fano foliations. 
\begin{proposition}[{Proposition~\ref{prop:loci_connected}}]
     Let $\mc F_W$ be a Fano toric foliation on a complete $\mb Q$-factorial toric variety $X_\Sigma$ of a fan $\Sigma$ in $N_\mb R$. 
     \begin{enumerate}
        \item $\on{dicrit}(\cF_W)$ is locally closed and any two irreducible components of $\overline{\on{dicrit}(\cF_W)}$ have non-empty intersection contained in $\on{dicrit}(\cF_W)$. 
        In particular, $\on{dicrit}(\cF_W)$ is connected. 
        \item  Any two irreducible components of $\on{Sing}(\cF_W)$ intersect. In particular, $\on{Sing}(\mc F_W)$ is connected.
     \end{enumerate}  
\end{proposition}

Last but not least, we investigate the interpolated $\delta$-lct, which is motivated by M\textsuperscript{c}Kernan's interpolated log canonical thresholds. 
In particular, we show the following theorems:
\begin{theorem}[{Theorem~\ref{thm:upper_lct_dense}}]
    Fix two positive integers $n$, $r$ with $r<n$ and a positive real number $\delta\leq\frac{1}{2}$. 
    Let $m=\lfloor\frac{1}{\delta}\rfloor$. 
    Then 
    $\dlctu_{n,r}^{\on{tor}}\cap [1-m\delta,1-\frac{m}{2m-1}\delta]$ is dense in $[1-m\delta,1-\frac{m}{2m-1}\delta]$. 
    In particular, $\dlctu_{n,r}^{\on{tor}}$ satisfies neither the ascending chain condition nor the descending chain condition. 
\end{theorem}
\begin{theorem}[{Theorem~\ref{thm:lower_delta_lct_dcc}}]
    $\dlctl_{n,r}^{\on{tor}}$ satisfies the descending chain condition.
\end{theorem}

\section*{Acknowledgments}
The authors would like to express their gratitude to Paolo Cascini, Ching-Jui Lai, and Calum Spicer for helpful discussions. 
The authors were partially supported by Lab. Bir. Geom. Grant number 111-2123-M-002-012-.

\section{Preliminaries}
We work over an algebraically closed field of characteristic zero. 

\subsection{Foliations}
In this subsection, most of the definitions follow from \cite{CS} and \cite{druel2021foliation}. 
Let $X$ be a normal variety of dimension $n$. 
A \emph{foliation} is a coherent subsheaf $\cF$ of the tangent sheaf $\cT_X$ such that
\begin{enumerate}
\item $\cF$ is saturated, that is $\cT_X/\cF$ is torsion-free, and
\item $\cF$ is closed under the Lie bracket.
\end{enumerate}

Let $r=\on{rank}(\cF)$ be the \emph{rank} of the foliation and $c=n-r$ be the \emph{corank} of the foliation. 
The \emph{canonical divisor} $K_\cF$ is a Weil divisor on $X$ such that $\cO_X(-K_\cF)\cong\det\cF$. 



Let $\pi\colon Y \dashrightarrow X$ be a dominant rational map between normal varieties and $\cF$ be a foliation on $X$. 
We denote by $\pi^{-1}\cF$ the \emph{pullback foliation} on $Y$ (see, for example, \cite[Section 3.2]{druel2021foliation}). 
If $f\colon X\dashrightarrow X'$ is birational, then $f_*\mathcal{F}$ represents the transformed foliation on $X'$ induced by $f^{-1}$.

Let $X^\circ$ be the open subset of $X$ such that $\cF\vert_{X^\circ}$ is a subbundle of $T_{X^\circ}$. 
A \emph{leaf} $L$ is a maximal connected and immersed holomorphic submanifold $L\subseteq X^\circ$ such that $T_L = \cF\vert_L$. 

A foliation $\cF$ is called \emph{algebraically integrable} if its leaves are algebraic. 
Equivalently, an algebraically integrable foliation $\cF$ on $X$ is induced from a dominant rational map $f\colon X\dashrightarrow Y$ for some normal variety $Y$ (see, for example, \cite[Sections 3.2 and 3.6]{druel2021foliation}).

\begin{definition}[Singular locus]
Let $\cF$ be a foliation of rank $r$ on a normal variety $X$. 
We obtain a morphism $\phi\colon (\Omega_X^r)^{**} \to \cO_X(K_\cF)$ by taking the double dual of the $r$-th wedge product of $\Omega_X^{**}\to \cF^*$, which is induced by the inclusion $\cF\subseteq \cT_X$. 
We define the \emph{singular locus} of $\cF$, denoted by $\on{Sing}(\cF)$, as the co-support of the image of $\phi'\colon (\Omega_X^r\otimes\cO_X(-K_\cF))^{**} \to \cO_X$. 
\end{definition}

\begin{definition}[Invariance]\label{defn_invariance}
Let $\cF$ be a foliation of rank $r$ on a normal variety $X$. 
\begin{enumerate}
    \item We say that a subvariety $S\subseteq X$ is $\cF$-\emph{invariant} if for any subset $U\subseteq X$ and any section $\partial\in H^0(U,\cF)$, we have 
    $\partial(\cI_{S\cap U})\subseteq\cI_{S\cap U}$ where $\cI_{S\cap U}$ is the ideal sheaf of $S\cap U$. 
    \item For any prime divisor $D\subseteq Y \xrightarrow{\pi} X$ over $X$ where $\pi$ is a birational morphism, we define $\iota(D)=0$ if $D$ is $\pi^{-1}\cF$-invariant and $\iota(D)=1$ if $D$ is non-$\pi^{-1}\mc F$-invariant. 
    Note that $\iota(D)$ is independent of the choice of the birational morphism $\pi$ that extracts $D$.
\end{enumerate}
\end{definition}

\begin{definition}[{Dicriticality, \cite[Definition 4.6]{CC}}]
    Let $\cF$ be foliation of rank $r$ on a normal variety $X$ of dimension $n$. 
    We say $\cF$ is \emph{dicritical} if there is an exceptional divisor $E$ over $X$ which is not foliation invariant and whose center on $X$ has dimension at most $n-r-1$. 

    We say $\cF$ is \emph{non-dicritical} if it is not dicritical. 
\end{definition}

\subsection{Toric varieties}
In this paper, every toric variety is assumed to be normal. Our notations closely follow \cite{CLS11}. 
We recall some notations that are frequently used in this paper. 

Let $N\simeq \bZ^n$ be a lattice of rank $n$ and $M:=\on{Hom}(N,\mb Z)$ be its dual lattice. We write $N_\mb R:=N\otimes\mb R$ and $N_\mb C:=N\otimes\mb C$.
A \emph{fan} $\Sigma$ in $N_\bR$ is a finite collection of rational, strongly convex, polyhedral cones $\sigma\subseteq N_\bR$, such that each face $\tau$ of a cone $\sigma\in\Sigma$ belongs to $\Sigma$ and the intersection of any two cones in $\Sigma$ is a face of each. 
The \emph{support} of $\Sigma$ is defined as $|\Sigma|=\bigcup_{\tau\in\Sigma}\tau$. 
For any $k\in\bZ_{\geq 0}$, denote the set of all $k$-dimensional cones in $\Sigma$ by $\Sigma(k)$, and denote the set of all $k$-dimensional faces of $\sigma\in\Sigma$ by $\sigma(k)$.
We write $\tau\preceq\sigma$ when $\tau$ is a face of $\sigma$. 

For each cone $\sigma\in\Sigma$, the affine toric variety associated with $\sigma$ is  $U_{\sigma,N}=\on{Spec}\bC[\sigma^\vee\cap M]=\on{Spec}\bC[\chi^m\mid m\in\sigma^\vee\cap M]$ where $\sigma^\vee$ is the dual cone of $\sigma$. 
A cone $\sigma\in\Sigma$ is said to be \emph{smooth with respect to $N$} if the primitive generators of the rays in $\sigma(1)$ form part of a $\bZ$-basis for $N$ (or equivalently, $U_{\sigma,N}$ is smooth).
If $\tau\preceq \sigma$ are two cones in $\Sigma$, there is an open immersion $U_{\tau,N}\hookrightarrow U_{\sigma, N}$.
The toric variety $X_{\Sigma,N}$ of the fan $\Sigma$ is constructed by gluing all $U_{\sigma,N}$ together via $\Sigma$. 
The dense torus $U_{\{0\},N}=\on{Spec}\bC[M]\subseteq X_{\Sigma,N}$ is denoted by $T_N$. 
The action of $T_N$ on itself can be extended to an action on $X_{\Sigma,N}$. 
We will omit $N$ in the subscript when $N$ is clear. 

For each $\sigma\in\Sigma$, $\on{Relint}(\sigma)$ denotes the relative interior of $\sigma$, $O_\sigma$ denotes the $T$-orbit of the distinguished point $x_\sigma$, and $V_\sigma$ denotes the closure of $O_\sigma$ in $X_\Sigma$ (see \cite[Chapter 3]{CLS11} for further details). 
If $\rho\in\Sigma(1)$ is a ray, then $V_\rho$ is a divisor and will also be denoted by $D_\rho$.

\begin{lemmadef}
    Let $D=\sum_{\rho\in\Sigma(1)} a_\rho D_\rho$ be a $\bR$-Cartier torus-invariant divisor on $X_\Sigma$ where $a_\rho\in\bR$ for all $\rho\in\Sigma(1)$. 
    Then there exists a function $\phi_D\colon |\Sigma| \to \bR$, called the \emph{support function} of $D$, which is linear on each cone $\tau\in\Sigma$ and $\phi_D(v_\rho) = -a_\rho$ for all $\rho\in\Sigma(1)$ where $v_\rho$ is the primitive element of $\rho$. 
\end{lemmadef}
\begin{proof}
    When $D$ is Cartier, it is \cite[Theorem 4.2.12]{CLS11}. 
    In general, we can write $D = \sum_{i=1}^k d_iD_i$ where $D_i$ are Cartier torus-invariant divisors and $d_i\in \bR$. 
    Let $\phi_i\colon |\Sigma|\to \bR$ be the support function associated with $D_i$. 
    Then we define $\phi_D\colon |\Sigma|\to \bR$ as $\sum_{i=1}^k\phi_i$. 
    It is clear that $\phi_D$ is linear on each cone $\tau\in\Sigma$ and $\phi_D(v_\rho) = -a_\rho$ for all $\rho\in\Sigma(1)$. 
    
    Note that such function $\phi_D$ is unique if it exists as it is linear on each cone $\tau\in\Sigma$. 
\end{proof}

\begin{theorem}
    Let $D$ be a torus-invariant $\bR$-Cartier divisor on a $\bQ$-factorial toric variety $X_\Sigma$. 
    The following statements are equivalent:
    \begin{enumerate}
        \item $D$ is ample.
        \item $\phi_D$ is strictly convex with respect to $\Sigma$, that is, $\phi_D(u+v) > \phi_D(u)+\phi_D(v)$ for all $u$, $v$ not in the same cone of $\Sigma$. 
        \item $\phi_D(u_{\rho_1}+\cdots+u_{\rho_k}) > \phi_D(u_{\rho_1})+\cdots+\phi_D(u_{\rho_k})$ for all primitive collections $P=\{\rho_1,\ldots,\rho_k\}$ of $\Sigma$. We recall that $P$ is a primitive collection if $P$ is not contained in $\sigma(1)$ for all $\sigma\in\Sigma$ but any proper subset is contained in $\sigma(1)$ for some $\sigma\in\Sigma$. 
    \end{enumerate}
\end{theorem}
\begin{proof}
    When $D$ is Cartier, it is well-known. (See \cite[Theorem 6.1.14 and Theorem 6.4.9]{CLS11}). 
    
    \begin{itemize}
        \item (1) $\Rightarrow$ (2). 
        If $D$ is ample, then we can write $D=\sum d_iD_i$ as $\bR$-linear combination of ample Cartier divisor where $d_i\in\bR_{>0}$. 
        Then $\phi_D(u+v) = \sum d_i\phi_{D_i}(u+v)>\sum d_i\phi_{D_i}(u)+\sum d_i\phi_{D_i}(v) = \phi_D(u)+\phi_D(v)$. 
        This shows that (1) implies (2). 
        
        \item (2) $\Rightarrow$ (3). 
        Let $\{\rho_1,\ldots,\rho_k\}$ be a primitive collection of $\Sigma$. 
        Note that $v:=u_{\rho_2}+\cdots+u_{\rho_k}\in\on{Relint}(\on{Cone}(\rho_2,\ldots,\rho_k))$. 
        Suppose $u_{\rho_1}$ and $v$ are on the cone $\tau\in\Sigma$. 
        Then $\rho_1\preceq\tau$ and $\on{Cone}(\rho_2,\ldots,\rho_k)\preceq\tau$. 
        Thus $\rho_2,\ldots,\rho_k\preceq\tau$. 
        As $\tau$ is simplicial, we have $\on{Cone}(\rho_1,\ldots,\rho_k)\in\Sigma$, which contradicts to that $\{\rho_1,\ldots,\rho_k\}$ is a primitive collection of $\Sigma$. 
        Hence, $u_{\rho_1}$ and $v$ are on the different cone of $\Sigma$. 
        Therefore, 
        \[\phi_D(u_{\rho_1}+\cdots+u_{\rho_k}) > \phi_D(u_{\rho_1}) + \phi_D(u_{\rho_2}\cdots+u_{\rho_k}) = \phi_D(u_{\rho_1})+\cdots+\phi_D(u_{\rho_k})\] 
        where the inequality follows from the strict convexity of $\varphi_D$. 
        This shows that (2) implies (3). 

        \item (3) $\Rightarrow$ (1). Suppose (3), we have $D$ is ample by \cite[Theorem 6.4.11 and Proposition 6.4.1]{CLS11}. 
    \end{itemize}
\end{proof}

\subsection{Toric foliations}
We recall the definition of toric foliations and their properties, for more details, see \cite{CC}. 
Let $X=X_\Sigma$ be the toric variety defined by a fan $\Sigma$ in $N_\bR$. 
A subsheaf $\cF\subseteq \cT_X$ is called $T$-\emph{invariant} or \emph{torus-invariant} if for any $t\in T$ we have $t^*\mc F=\mc F$ as subsheaves under the natural isomorphism $t^*\cT_X\simeq \cT_X$. 
A foliation $\cF\subseteq \cT_X$ is called a \emph{toric foliation} if $\cF$ is $T$-invariant. 

We recall the following description and properties for toric foliations: 
\begin{proposition}[{\cite[Proposition 3.1]{CC} or \cite[Lemma 2.1.5]{pang2015harder}}]\label{1-1}
Let $\Sigma$ be a fan in $N_\bR$ and $X_{\Sigma,N}$ be the toric variety defined by $\Sigma$. 
Then there is a one-to-one correspondence between the set of toric foliations on $X_{\Sigma,N}$ and the set of complex vector subspaces $W\subseteq N_\bC$. 
\end{proposition}

We will use $\cF_{W,\Sigma,N}$ to denote the toric foliation on $X_{\Sigma,N}$ corresponding to the complex vector subspace $W\subseteq N_\bC$. 
If we have another fan $\Sigma'$ in the same $N_\mb R$, the transformed foliation on $Y = X_{\Sigma',N}$ is nothing but $\cF_{W,\Sigma',N}$. Hence we can unambiguously write $\mc F_W$ to denote the transformed foliation on any birational model obtained by modifying the defining fan.   

\begin{proposition}[{\cite[Corollary 3.3]{CC}}]
    Let $\cF_W$ be a toric foliation on a toric variety $X_\Sigma$. 
    For any $\rho\in\Sigma(1)$, $D_\rho$ is invariant if and only if $\rho\nsubseteq W$. 
\end{proposition}

\begin{proposition}[{\cite[Proposition 3.7]{CC}}]
    Let $\cF_W$ be a toric foliation on a toric variety $X_\Sigma$. 
    Then we have $K_\cF\sim -\sum_{\rho\in\Sigma(1),\,\rho\subseteq W} D_\rho$. 
\end{proposition}

\section{Adjoint foliated structure}
As \cite[Example 3.9]{CJV24} provides an unbounded example for the Fano foliations when both $X$ and $\cF$ have at worst log canonical singularities. 
Then we would like to impose $\varepsilon$-log canonicity on foliation singularities. 
However, when $\cF$ is algebraically integrable or toric, the set consisting of $\varepsilon$-log canonical Fano foliations is empty. 
It is therefore natural to consider the adjoint foliated structure. 

\begin{definition}[{\cite{M7}}]
    We say $(X,\cF,\Delta,t)$ is an \emph{adjoint foliated structure} if the following conditions hold:
\begin{enumerate}
    \item $X$ is a normal projective variety, 
    \item $\cF$ is a foliation on $X$, 
    \item $\Delta$ is an effective $\bR$-divisor, 
    \item $t\in [0,1]$, and
    \item the adjoint canonical divisor $K_{(X,\cF,\Delta,t)} = tK_\cF + (1-t)K_X+\Delta$ is $\bR$-Cartier.
\end{enumerate}
If $\Delta=0$, then we will denote the foliated structure $(X,\cF,0,t)$ as the triple $(X,\cF,t)$. 
\end{definition}

Let $\pi\colon X'\to X$ be a birational morphism and $\cF'=\pi^{-1}\cF$. 
We define $\Delta'$ as the unique $\bR$-divisor such that 
\[K_{(X',\cF',\Delta',t)} = \pi^*K_{(X,\cF,\Delta,t)}.\]
For any prime divisor $E$ on $X'$, we define the \emph{adjoint discrepancy} to be $-\on{mult}_E\Delta'$ and the \emph{adjoint log discrepancy} to be $a(E,\cF,\Delta,t)=-\on{mult}_E\Delta'+t\iota(E)+(1-t)$. 
For $\delta\in [0,1]$, we say $(X,\cF,\Delta,t)$ is \emph{$\delta$-log canonical} or \emph{$\delta$-lc} if $a(E,\cF,\Delta,t)\geq (t\iota(E)+(1-t))\delta$ for any divisor $E$ over $X$. 

\begin{remark}
For $t\in (0,1]$, we note that $(X,\cF,\Delta,t)$ is $\delta$-lc if and only $(X,\cF,B)$ is $(\varepsilon=\frac{1-t}{t},\delta)$-adjoint lc in the sense of \cite[Definition 2.12]{spicer2023effective} where $B=\Delta_{\textnormal{n-inv}}+\frac{1}{1-t}\Delta_{\textnormal{inv}}$ for $t\neq 1$ and $B=\Delta_{\textnormal{n-inv}}$ for $t=1$. 
\end{remark}

The following Proposition provides an evidence that the definition of adjoint log discrepancy is canonical:
\begin{proposition}\label{prop:adjoint_log_discrepancy}
    Let $\cF=\cF_W$ be a toric foliation on a $\bQ$-factorial toric variery $X=X_\Sigma$, where $\Sigma$ is a fan in $N_\bR$ and  $W\subset N_\bC$ is a linear subspace. 
    Let $\Delta$ be a torus-invariant divisor on $X$ and $\Sigma'$ be a refinement of $\Sigma$ and $\pi\colon X_{\Sigma'}\to X$ be the corresponding toric birational morphism. 
    Then for any $\rho\in\Sigma'(1)\setminus\Sigma(1)$, the adjoint log discrepancy $a(D_\rho,\cF_W,\Delta,t)=\phi_{K_{(X,\cF_W,\Delta,t)}}(v_\rho)$ where $v_\rho$ is the primitive generator of $\rho$. 
\end{proposition}
\begin{proof}
    Let $\Delta'$ be an $\bR$-divisor on $X'$ such that $K_{(X',\cF',\Delta',t)} = \pi^*K_{(X,\cF,\Delta,t)}$ where $\cF'=\pi^{-1}\cF$. 
    Note that $\on{mult}_{D_\rho}K_{\cF'} = \iota(D_\rho)$ and $\on{mult}_{D_\rho}K_{X'}=1$. 
    Then 
    \begin{align*}
        a(D_\rho,\cF_W,\Delta,t) &= -\on{mult}_{D_\rho}\Delta' + t\iota(D_\rho) +(1-t) \\
        &= -\on{mult}_{D_\rho}(\pi^*K_{(X,\cF,\Delta,t)} - tK_{\cF'} - (1-t)K_{X'}) + t\iota(D_\rho) +(1-t) \\
        &= -\on{mult}_{D_\rho}(\pi^*K_{(X,\cF,\Delta,t)}) \\
        &= \phi_{(K_{X},\cF,\Delta,t)}(v_\rho).
    \end{align*}
\end{proof}

The following lemma gives a combinatorial criterion of $\delta$-log canonicity for the toric adjoint foliated structures: 
\begin{lemma}\label{lem:delta_lc_combinatorial}
    Let $\cF=\cF_W$ be a toric foliation on a toric variery $X=X_\Sigma$, where $\Sigma$ is a fan in $N_\bR$ and  $W\subseteq N_\bC$ is a linear subspace. 
    Let $\Delta$ be a torus-invariant divisor on $X$. 

    Then the adjoint foliated structure $(X,\cF,\Delta,t)$ is $\delta$-lc if and only if 
    \[\phi_{K_{(X,\cF,\Delta,t)}}(v) \geq\begin{cases}
        (1-t)\delta & \mbox{if } v\notin W \\
        \delta & \mbox{if } v\in W
    \end{cases}\]
    for all primitive elements $v\in|\Sigma|$. 
\end{lemma}
\begin{proof}
    This is a special case of Lemma~\ref{lem:toroidal_delta_lc_combinatorial}. 
\end{proof}
\begin{remark}
    Even if we want to give a direct proof for the toric case, it is inevitable to consider the toroidal case in the appendix as we need to check the adjoint log discrepancies for all exceptional divisors, no matter whether they are torus-invariant or not.
\end{remark}

\section{Boundedness of Fano ajoint foliated structures}
In this section, we will prove the main Theorem~\ref{thm:bounded_adjoint_Fano}. 
Before we state and prove the main Theorem, we provide some examples to demonstrate some properties of Fano foliations. 
\begin{example}[Fano foliations which are not algebraically integrable]
    Let $N=\bZ e_1\oplus\bZ e_2\oplus\bZ e_3$ and $v=-e_1-e_2-e_3$. 
    We consider the complete fan $\Sigma$ with $\Sigma(1) = \{\bR_{\geq 0}e_1, \bR_{\geq 0}e_2, \bR_{\geq 0}e_3, \bR_{\geq 0}v\}$. 
    Note that the toric variety $X_\Sigma=\bP^3$. 
    Let $W_a=\bC e_1+\bC(e_2+ae_3)$ where $a\notin\bQ$. 
    Then $W_a\neq \bC(W_a\cap N)=\bC e_1$ and thus $\cF_{W_a}$ is not algebraically integrable by \cite[Proposition 3.14]{CC}. 
    Moreover, $-K_{\cF_{W_a}} = D_{\bR_{\geq 0}e_1}$ is ample and therefore, $\cF_{W_a}$ is a Fano foliation. 
\end{example}

\begin{example}[Fano foliations on non-Fano smooth varieties]
    Let $s\in\bN$, $u_1=(0,s,1)$, $u_2=(0,s,-1)$, $u_3=(-1,1,0)$, $u_4=(1,0,0)$, $u_5=(0,-1,0)$. 
    Let $\rho_i = \bR_{\geq 0}u_i$ for $i\in\{1,\ldots,5\}$. 
    Let $W = yz$-plane and 
    \[\Sigma(3)=\{\sigma_{134}, \sigma_{234}, \sigma_{135}, \sigma_{145}, \sigma_{235}, \sigma_{245}\}\] 
    where $\sigma_{ijk}=\operatorname{Cone}(u_i,u_j,u_k)$. 
    It is straightforward to see that $X_\Sigma$ is smooth. 

    Note that there are only two primitive collections, $\{\rho_1, \rho_2\}$ and $\{\rho_3,\rho_4,\rho_5\}$. 

    Let $\phi_\cF := \phi_{-K_\cF}$ and $\phi_X := \phi_{-K_X}$. 
    Note that $\varphi_X(u_i)=-1$ for all $i\in\{1,\ldots,5\}$ and 
    \[\varphi_\cF(u_1) = \varphi_\cF(u_2) = \varphi_\cF(u_5) = -1 \mbox{ and } \varphi_\cF(u_3) = \varphi_\cF(u_4) = 0.\]

    We study two primitive collections: 
    \begin{itemize}
        \item For $\{\rho_1,\rho_2\}$, we have $u_1+u_2 = (0,2s,0) = 2su_3+2su_4$. 
        Thus, 
        \begin{align*}
            \varphi_\cF(u_1+u_2) &= \varphi_\cF(2su_3+2su_4) = 0 > -2 = \varphi_\cF(u_1)+\varphi_\cF(u_2) \mbox{ and } \\
            \varphi_X(u_1+u_2) &= \varphi_X(2su_3+2su_4) = -4s \leq -2 = \varphi_X(u_1)+\varphi_X(u_2).
        \end{align*}

        \item For $\{\rho_1,\rho_2\}$, we have $u_3+u_4+u_5 = (0,0,0)$. 
        Thus, 
        \begin{align*}
            \varphi_\cF(u_3+u_4+u_5) &= \varphi_\cF(0) = 0 > -1 = \varphi_\cF(u_3) + \varphi_\cF(u_4) + \varphi_\cF(u_5) \mbox{ and } \\
            \varphi_X(u_3+u_4+u_5) &= \varphi_X(0) = 0 > -3 = \varphi_X(u_3) + \varphi_X(u_4) + \varphi_X(u_5). 
        \end{align*}
    \end{itemize}
    Therefore, $-K_\cF$ is ample but $-K_X$ is not ample by \cite[Theorem 6.4.9]{CLS11}. 
\end{example}

The following lemma shows that there is no $\delta$-lc Fano toric foliations: 
\begin{lemma}\label{lem:Fano_fol_exc_ld_zero}
    Let $\cF:=\cF_W$ be a Fano toric foliation on a complete $\bQ$-factorial toric variety $X_\Sigma$. 
    If $\cF\neq\cT_X$, then there is a non-foliation-invariant exceptional divisor $E$ with foliated log discrepancy zero. 
\end{lemma}
\begin{proof}
    Let $W'=\operatorname{span}_\bQ\{\rho\in\Sigma(1)\mid\rho\subseteq W\}$ and $\phi\colon N_\bR\to\bR$ be the support function of $-K_\cF$. 
    Note that $\phi(u)\leq 0$ for all $u\in N_\bR$. 
    Moreover, $\phi$ is strictly convex with respect to $\Sigma$ as $-K_\cF$ is ample. 

    Since $\cF_W$ is Fano, we have $\cF_W\neq 0$ and thus there is a ray $\rho\in\Sigma(1)$ such that $\rho\subseteq W$. 
    Hence $r':=\dim_\bQ W'\geq 1$. 
    As $\cF\neq\cT_X$, we have $r'\leq\dim_\bC W < n$. 
    Let $\pi\colon N_\bQ \to N_\bQ/W'$ be the canonical linear projection. 
    Note that $|\Sigma|=N_\bR$ as $X_\Sigma$ is complete. 
    Since $r'<n$, we have $\dim_\bQ N_\bQ/W' = n-r'\geq 1$ and thus the set 
    \[S:=\{\rho\in\Sigma(1)\mid \pi(v_\rho)\neq 0\} = \{\rho\in\Sigma(1)\mid\rho\nsubseteq W\}\]
    is not empty. 
    Then $N_\bQ/W' = \sum_{\rho\in S} \bQ_{\geq 0}\pi(v_\rho)$ where $v_\rho$ is the primitive generator of $\rho$. 
    Hence $\# S\geq\dim_\bQ N_\bQ/W'+1\geq 2$. 
    Therefore, there are $a_\rho\in\bQ_{>0}$ for $\rho\in S$ such that $\sum_{\rho\in S} a_\rho\pi(v_\rho)=0$. 
    Let $v_0:=\sum_{\rho\in S} a_\rho v_\rho\in N_\bQ$. 
    Since $\pi(v_0)=0$, we have $v_0\in W'\cap N_\bQ\subseteq W\cap N_\bQ$. 
    

    Moreover, we note that 
    \[0 \geq\phi(v_0) = \phi(\sum_{\rho\in S} a_\rho v_\rho) \geq \sum_{\rho\in S} a_\rho\phi(v_\rho) = 0\]
    where the first inequality holds since $\phi\leq 0$, the second inequality holds because of the strict convexity of $\phi$, and the last equality follows as $\rho\nsubseteq W$. 
    Hence, we have $\phi(\sum_{\rho\in S} a_\rho v_\rho) = \sum_{\rho\in S} a_\rho\phi(v_\rho)$ and therefore, all $v_\rho$, for $\rho\in S$, are in the same maximal cone because $\phi$ is strictly convex. 
    Since all cones in $\Sigma$ are strictly convex, $v_0=\sum_{\rho\in S} a_\rho v_\rho\neq 0$ and the ray generated by $v_0$ is not in $\Sigma(1)$. 
    So the star subdivision along $v_0$ introduces an exceptional divisor which is not foliation-invariant and has log discrepancy zero. 
\end{proof}

\begin{proposition}\label{prop:delta_lc_imply_TX}
    Let $\cF_W$ be a Fano toric foliation on a complete $\bQ$-factorial toric variety $X_\Sigma$. 
    If the adjoint foliated structure $(X,\cF,1)$ is $\delta$-lc for some $\delta>0$, then $\cF=\cT_X$. 
\end{proposition}
\begin{proof}
    Suppose $\cF\neq\cT_X$, then by Lemma~\ref{lem:Fano_fol_exc_ld_zero}, there is a non-foliation-invariant exceptional divisor $E$ such that $a(E,\cF)=0$. 
    Then the adjoint log discrepancy $a(E,X,\cF,1) = a(E,\cF)=0$ and thus $(X,\cF,1)$ is not $\delta$-lc, which is a contradiction. 
    Therefore, $\cF=\cT_X$. 
\end{proof}

\begin{theorem}[Main Theorem]\label{thm:bounded_adjoint_Fano}
    Fix a positive integer $n$, a non-negative integer $r\leq n$, a positive real number $\delta$, and two real numbers $t_i\in [0,1)$ for $i=1$, $2$. 
    Let $\cF=\cF_W$ be a toric foliation of rank $r$ on a complete toric variety $X=X_\Sigma$ of dimension $n$, $\Delta$ be an effective torus-invariant divisor on $X_\Sigma$, and $K_{t_i} = t_iK_\cF+(1-t_i)K_X+\Delta$ for $i=1$, $2$. 
    Suppose $-K_{t_1}$ is ample and the adjoint foliated structure $(X,\cF,\Delta,t_2)$ is $\delta$-lc. 
    Then we have the following:
    \begin{enumerate}
        \item There is an effective divisor $\Delta'$ such that $K_X+\Delta'$ is $\bR$-Cartier. 
        \item $(X,\Delta')$ is $(1-t_2)\delta$-lc. In particular, $X$ is potentially klt. 
        \item The number of isomorphism classes of such $X$ is finite, which depends only on $t$, $\delta$, and $n$. 
        \item The set consisting of $(X,\cF)$ is bounded. 
    \end{enumerate}
\end{theorem}
\begin{proof}
    \begin{enumerate}
        \item We recall that $K_X=-\sum_{\rho\in\Sigma(1)}D_\rho$. 
        We write $K_{t_2}=-\sum_{\rho\in\Sigma(1)}c_\rho D_\rho$ where $c_\rho\in [0,1]$. 
        Let $\Delta'=K_{t_2}-K_X$. 
        Note that $\Delta'$ is effective as $1-c_\rho\geq 0$. 

        \item Note that $\phi_{K_X+\Delta'}(v) = \phi_K(v)\geq \phi_{K_{t_2}}(v)\geq (1-t_2)\delta$ for all $v\in N$ and thus $(X,\Delta')$ is $(1-t_2)\delta$-lc. 

        \item Let $v_\rho$ be the primitive generator of $\rho$ for $\rho\in\Sigma(1)$. 
    Let $P=\operatorname{conv}(v_\rho\mid \rho\in\Sigma(1))$ and $\widetilde{P} = \operatorname{conv}((-\on{mult}_{D_\rho}K_{t_1})v_\rho\mid\rho\in\Sigma(1))\subseteq \operatorname{conv}((1-t_1)v_\rho, v_{\rho'}\mid \rho\not\subseteq W, \rho'\subseteq W)$. 

    We define $\iota(\rho)=1$ if $\rho\subseteq W$ and $\iota(\rho)=0$ otherwise. 

    Since $(X,\cF,\Delta,t_2)$ is $\delta$-lc, by Lemma~\ref{lem:delta_lc_combinatorial}, we have 
    \begin{align*}
        t_2\iota(\rho)+(1-t_2)-\on{mult}_{D_\rho}\Delta 
        &=-\on{mult}_{D_\rho}K_{t_2} \\
        &=\phi_{K_{t_2}}(v_\rho) \\
        &\geq (1-t_2)\delta
    \end{align*}
    and thus, $\on{mult}_{D_\rho}\Delta\leq t_2\iota(\rho)-(1-t_2)(1-\delta)$. 
    Hence, 
    \begin{align*}
        -\on{mult}_{D_\rho}K_{t_1} &= t_1\iota(\rho)+(1-t_1)-\on{mult}_{D_\rho}\Delta \\
        &\geq (t_1-t_2)\iota(\rho)+(1-t_1)+(1-t_2)\delta \\
        &\geq \min\{(1-t_1)+(1-t_2)\delta, (1-t_2)(1+\delta)\}=:\lambda. 
    \end{align*}
    Therefore, $\lambda P\subseteq\widetilde{P}$ as $0\in P\cup \widetilde{P}$. 
    For each maximal cone $\sigma\in\Sigma$, we define $P^\sigma = \operatorname{conv}(0,(-\on{mult}_{D_\rho}K_{t_1})v_\rho\mid\rho\in\sigma(1))$ and thus 
    \[\lambda P\cap\sigma\subseteq \widetilde{P}\cap\sigma\subseteq P^\sigma\] 
    where the second inclusion follows from ampleness of $-K_{t_1}$. 
    Since $(X,\Delta')$ is $(1-t_2)\delta$-lc, we have that $\lambda(1-t_2)\delta P^\sigma$ contains no lattice point other than the origin for any $\sigma\in\Sigma(n)$. 
    Thus 
    \[\lambda(1-t_1)(1-t_2)\delta P\cap \sigma = \lambda(1-t_2)\delta((1-t_1)P\cap\sigma)\subseteq \lambda(1-t_2)\delta P^\sigma\] contains no lattice point other than the origin for any $\sigma\in\Sigma(n)$. 
    Hence, $\lambda(1-t_1)(1-t_2)\delta P$ contains no lattice point other than the origin. 
    Therefore, by \cite[Proposition 5]{borisov1993singular}, there are finitely many $P$ up to the action of $\on{GL}_n(\bZ)$.  
    
    Note that $P$ contains all primitive generators of $\rho$ for $\rho\in\Sigma(1)$. 
    So there are only finitely many possible $\Sigma(1)$ and thus there are only finitely many possible $\Sigma$. 

    \item For each $X$, the toric foliation on $X$ is parametrized by $\on{Gr}(r,n)$. Since there are only finitely many such $X$, the set consisting of $(X,\cF)$ is bounded. 
    \end{enumerate}
\end{proof}



\section{Dicritical locus}
In this section, we study the locus where $\cF$ is dicritical and show that it is connected when $\cF$ is Fano. 
\begin{definition}
    Let $\cF$ be a foliation of rank $r$ on a normal variety $X$ of dimension $n$. 
    We define the \emph{dicritical locus}, $\on{dicrit}(\cF)$ to be the union of subvarieties $Z$ of dimension at most $n-r-1$ such that there is a non-$\cG$-invariant divisor $E$ over $Z$ where $\cG$ is the transformed foliation. 
\end{definition}

\begin{definition}
    $(\tau,W)$ is dicritical if $\on{Relint}(\tau)\cap W\cap N\neq\emptyset$ and $\tau\nsubseteq W$. 
\end{definition}

\begin{lemma}\label{lem:Fano_supp_fcn_zero_locus}
    Let $\Sigma$ be a complete simplicial fan in $N_\bR$, and let $\phi\colon N_\bR\to\bR$ be a non-positive function which is strictly convex and piecewise linear with respect to $\Sigma$. 
    Then there exists a cone $\tau_0\in\Sigma$ such that $\phi(v)=0$ if and only if $v\in\tau_0$. 
\end{lemma}
\begin{proof}
    Note that if there exists $v\in\on{Relint}(\tau)$ for some $\tau\in\Sigma$ such that $\phi(v)=0$, then $\phi\vert_\tau=0$ by the linearity of $\phi$. 
    Let $S = \{\tau\in\Sigma :\phi\vert_\tau=0\}$. 
    Note that $S$ is not empty as the zero cone $\{0\}$ belongs to $S$. 
    Then by Zorn's lemma, we have a cone $\tau_0\in\Sigma$, a maximal element in $S$ with respect to the inclusion. 

    Suppose there is an element $v$ such that $\phi(v)=0$ but $v\notin\tau_0$. 
    Then $v\in\on{Relint}(\sigma)$ for some $\sigma\in\Sigma$. 
    So $\sigma\in S$. 
    Since $\tau_0$ is maximal, $\on{Cone}(\sigma,\tau_0)\notin\Sigma$. 
    Fix en element $u\in\on{Relint}(\tau_0)$. 
    Note that $0\geq\phi(u+v)>\phi(u)+\phi(v)=0$, which is impossible. 
\end{proof}

\begin{lemma}\label{lem:dicrit_locus}
    Let $\cF_W$ be a toric foliation on a toric variety $X_\Sigma$ of dimension $n$. 
    Then 
    \[\on{dicrit}(\mc F_W)=\begin{cases}
        \bigcup_{\tau\in\Sigma,\,(\tau,W)\ \textnormal{is dicritical}}\mc O_\tau & \mbox{if } \on{rank}\mc F\geq n-1 \\
        \bigcup_{\tau\in\Sigma,\,(\tau,W)\ \textnormal{is dicritical}}V_\tau & \mbox{if } \on{rank}\mc F< n-1.
    \end{cases}\]
    We use the convention that the union over an empty index set is an empty set. 
\end{lemma}

It is clear when $\cF=0$. 
If $\on{rank}\cF=n$, then $\cF=\cT_X$ is non-dicritical and $(\tau,W)$ is not dicritical for any $\tau\in\Sigma$. 
For the proof when $1\leq\on{rank}\cF\leq n-1$, it involves the language of the toroidal structures. 
So we leave the proof in the appendix. 

Using Lemma~\ref{lem:dicrit_locus}, we show the following properties for Fano toric foliations:
\begin{proposition}\label{prop:loci_connected}
     Let $\cF_W$ be a Fano toric foliation on a complete $\bQ$-factorial toric variety $X_\Sigma$ of a fan $\Sigma$ in $N_\mb R$. 
     \begin{enumerate}
        \item $\on{dicrit}(\cF_W)$ is locally closed and any two irreducible components of $\overline{\on{dicrit}(\cF_W)}$ have non-empty intersection contained in $\on{dicrit}(\cF_W)$. 
        In particular, $\on{dicrit}(\cF_W)$ is connected. 
        \item  Any two irreducible components of $\on{Sing}(\cF_W)$ intersect. In particular, $\on{Sing}(\mc F_W)$ is connected.
     \end{enumerate} 
\end{proposition}
\begin{proof}
\begin{enumerate}

    \item The local closedness of $\on{dicrit}(\cF_W)$ follows immediately from Lemma~\ref{lem:dicrit_locus}. 

    To show any two irreducible components of $\overline{\on{dicrit}(\cF_W)}$ have non-empty intersection contained in $\on{dicrit}(\cF_W)$, it suffices to show that if $V_{\tau_1}$ and $V_{\tau_2}$ are two irreducible components of $\overline{\on{dicrit}(\cF_W)}$, then there exists $\cO_{\widetilde{\tau}}\subseteq\on{dicrit}(\cF_W)$ such that $\tau_1$, $\tau_2\preceq\widetilde\tau$.
    
    Let $\phi=\phi_{-K_{\mc F_W}}$ be the support function for $-K_{\cF_W}$ and
    $\tau_1$, $\tau_2\in\{\tau\in\Sigma\mid (\tau,W)\ \text{is dicritical}\}$ be two minimal elements. 
    \begin{claim}
        $W\cap\tau_i$ does not contain any ray in $\tau_i(1)$ for $i=1$, $2$.
    \end{claim}
    \begin{claimproof}
        Suppose not, then there is a ray $\rho_1\in\tau_1(1)$ such that $\rho_1\subseteq W$. 
        Let $\underline{\tau_1} = \on{Cone}(\rho\mid\rho\in\tau_1(1)\setminus\{\rho_1\})$. 
        Note that $\underline{\tau_1}\preceq\tau_1$ and $(\underline{\tau_1},W)$ is dicritical. 
        This contradicts to the minimality of $\tau_1$. 
        So there is no ray $\rho\in\tau_1(1)$ contained in $W$. 
        Similarly for $\tau_2$. 
        This completes the proof of the claim. 
    \end{claimproof}
    
    Thus, we have $\phi\vert_{\tau_i}=0$ for $i=1$, $2$. 
    By Lemma~\ref{lem:Fano_supp_fcn_zero_locus}, there is a cone $\tau_0\in\Sigma$ such that $\phi(v)=0$ if and only if $v\in\tau_0$. 
    Thus, we have $\tau_i\preceq \tau_0$ for $i=1$, $2$. 
    Since $X_\Sigma$ is $\bQ$-factorial, we have $\tau_0$ is simplicial and thus $\widetilde\tau := \on{Cone}(\tau_1,\tau_2)\in\Sigma$ because $\Sigma$ is a fan. 
    \begin{claim}
        $(\widetilde\tau,W)$ is dicritical. 
    \end{claim}
    \begin{claimproof}
        As $(\tau_i,W)$ is dicritical, there are $v_i\in\on{Relint}(\tau_i)\cap W\cap N$. 
        Then $v_1+v_2\in\on{Relint}(\widetilde\tau)\cap W\cap N$. 
        Moreover, we have $\widetilde\tau\nsubseteq W$ as $\tau_i\nsubseteq W$ for $i=1$, $2$. 
        Therefore, $(\widetilde\tau,W)$ is dicritical. 
        This completes the proof of the claim. $\quad\blacksquare$
    \end{claimproof}
    
    \item  Recall \cite[Proposition 3.9]{CC} that $\on{Sing}(\cF_W)=\bigcup_{\tau\in\Sigma,\,(\tau,W)\textnormal{ is singular}}V_\tau$ where $(\tau,W)$ is singular if $W\cap \bC\tau\neq\on{Span}_\bC(S)$ for any $S\subseteq\tau(1)$. 
    Let $\tau_1$, $\tau_2\in\{\tau\in\Sigma\mid(\tau,W)\ \text{is singular}\}$ be two minimal elements. 
    \begin{claim}
        $W\cap\tau_i$ does not contain any ray in $\tau_i(1)$ for $i=1$, $2$.
    \end{claim}
    \begin{claimproof}
        We write $\tau_1=\bR_{\geq 0}v_1+\cdots +\bR_{\geq 0}v_\ell$ where $\ell=\dim\tau_1$ and each $v_j\in N$ is primitive. 
        Since $(\tau_i,W)$ is singular, we have $W\cap(\bC v_1+\cdots+\bC v_\ell)\neq\on{Span}_\bC(S)$ for any $S\subseteq\{v_1\ldots,v_\ell\}$. 
        Suppose on the contrary that $v_1\in W$. Then we have $W\cap(\mb Cv_2+\cdots+\mb Cv_\ell)\neq\on{Span}(S)$ for any $S\subseteq\{v_2\ldots,v_\ell\}$, which means $(\on{Cone}(v_2,\ldots,v_\ell),W)$ is singular, contradicting the minimality of $\tau$. 
        Similarly we have $W\cap \tau_2$ does not contain any ray in $\tau_2(1)$. 
        This completes the proof of the claim. $\quad\blacksquare$
    \end{claimproof}
    
    As a result, we have $\phi\vert_{\tau_i}=0$. 
    By Lemma~\ref{lem:Fano_supp_fcn_zero_locus}, there is a cone $\tau_0\in\Sigma$ such that $\phi(v)=0$ if and only if $v\in\tau_0$. 
    Thus, we have  $\tau_i\preceq \tau_0$ for $i=1$, $2$. 
    Since $X_\Sigma$ is $\bQ$-factorial, we have $\tau_0$ is simplicial and thus $\widetilde\tau := \on{Cone}(\tau_1,\tau_2)\in\Sigma$ because $\Sigma$ is a fan. 

    Moreover, $W\cap\widetilde\tau$ does not contain any ray in $\widetilde\tau(1)$. 
    Hence $W\cap\bC\widetilde\tau\neq\on{Span}_\bC(S)$ for any $S\subseteq\widetilde\tau(1)$. 
    Therefore, $(\widetilde\tau,W)$ is singular. 
\end{enumerate}
\end{proof}

\begin{remark}
    Note that Proposition~\ref{prop:loci_connected} holds true for $\cF=\cT_X$ without Fano condition. 
    Indeed, let $X$ be a normal projective variety. 
    Note that $\cT_X$ is a foliation with $\operatorname{Sing}(\cT_X)=\emptyset$ by definition of singular locus of the foliation. 
\end{remark}

\begin{example}
\begin{enumerate} 
    \item ($\on{dicrit}(\cF_W)$ might not be Zariski closed in general.)
    
    Let $N=\bZ^4$ and $\Sigma$ be the complete fan generated by $v_1=(1,0,0,0)$, $v_2=(0,1,0,0)$, $v_3=(0,0,1,0)$, $v_4=(0,0,0,1)$, and $v_5=(-1,-1,-1,-1)$. 
    We have $X_\Sigma\simeq\mb P^4$. 
    Let $W=\on{Span}_\bC((0,1,0,0),(0,0,1,1),(\pi,0,1,0))$. 
    Then 
    $-K_{\cF_W}=\sum_{\rho\in\Sigma(1),\,\rho\subseteq W}D_\rho=D_{\bR_{\geq 0}v_2}$, which is ample. 
    Let $\sigma_{234}=\on{Cone}(v_2,v_3,v_4)$ and $\sigma_{1234}=\on{Cone}(v_1,v_2,v_3,v_4)$. 
    One can check that $(\sigma_{234},W)$ is dicritical and $(\sigma_{1234},W)$ is non-dicritical. 
    Hence $\mc O_{\sigma_{234}}\subseteq \on{dicrit}(\mc F_W)$ and $\mc O_{\sigma_{1234}}\not\subseteq \on{dicrit}(\mc F_W)$. 
    Since $\mc O_{\sigma_{1234}}\subseteq V_{\sigma_{234}}=\overline{\mc O_{\sigma_{234}}}$, we conclude that $\on{dicrit}(\mc F_W)$ is not Zariski closed. 

    \item ($\on{dicrit}(\mc F_W)$ and $\overline{\on{dicrit}(\mc F_W)}$ can be reducible in general.) 
    
    Let $N=\mb Z^3$ and $\Sigma$ be the fan generated by $v_1=(1,0,0),v_2=(0,1,0),v_3=(0,0,1)$ and $v_4=(-1,-1,-1)$. We have $X_\Sigma\simeq\mb P^3$. Consider $W=\on{Span}_{\mb C}((2,0,1),(0,2,1))\subseteq N_\mb C$. Then $-K_{\mc F_W}=\sum_{\rho\in\Sigma(1),\rho\subseteq W}D_\rho=D_{\bR_{\geq 0}v_4}$, which is ample. 
    Let $\sigma_{ij}=\on{Cone}(v_i,v_j)$ and $\sigma_{ijk}=\on{Cone}(v_i,v_j,v_k)$. 
    One can check that the cone $\tau\in\Sigma$ such that $(\tau,W)$ is dicritical are $\sigma_{123}$, $\sigma_{134}$, $\sigma_{234}$, $\sigma_{13}$, and $\sigma_{23}$. 
    Hence, $\overline{\on{dicrit}(\mc F_W)}=\on{dicrit}(\mc F_W)=V_{\sigma_{13}}\cup V_{\sigma_{23}}$, which is a union of two curves. 
    Note that in this case $\on{dicrit}(\cF_W)=\on{Sing}(\cF_W)$.
\end{enumerate}  
\end{example}

In general, we also show the irreducibility of the dicritical loci and singular loci for Fano foliations, which are not necessarily toric, of rank one on normal projective varieties. (See Proposition~\ref{prop:Fano_rank_one_general}.)

\section{\texorpdfstring{Interpolated $\delta$-lct}{Interpolated delta-lct}}
In this section, instead of usual ($\delta$-)log canonical thresholds, we study the interpolated $\delta$-log canonical thresholds. 



\begin{definition}
    Let $\mc F$ be a foliation on a normal variety $X$, and let $\Delta\in[0,1]$ be an $\mb R$-divisor on $X$ such that $(X,\cF,\Delta,t)$ is an adjoint foliated structure. 
    For any $\delta>0$, we define two interpolated $\delta$-lct as follows:
    \begin{align*}
        \dlctl(X,\mc F,\Delta)&:=\inf\{t\in[0,1]\mid(X,\mc F,\Delta,t)\ \text{is}\ \delta\text{-lc}\}\ \text{and}\\
        \dlctu(X,\mc F,\Delta)&:=\sup\{t\in[0,1]\mid(X,\mc F,\Delta,t)\ \text{is}\ \delta\text{-lc}\}.
    \end{align*}
\end{definition}
\begin{remark}
    The interpolated lct was first introduced M\textsuperscript{c}Kernan as $0$-$\overline{\on{LCT}}(X,\cF,\Delta=0)$. 
\end{remark}

\begin{definition} 
    Let $\delta$ be a non-negative real number. 
    We define $\dlctl_{n,r}^{\on{tor}}$ (resp. $\dlctu_{n,r}^{\on{tor}}$) to be the set of $\dlctl(X,\cF)$ (resp. $\dlctu(X,\cF)$) where $(X,\cF)$ runs through all couples such that $\cF$ is a toric foliation of rank $r$ on a $\bQ$-factorial toric variety $X$ of dimension $n$. 
 \end{definition}

We first show the following theorem that $\dlctu_{n,r}^{\on{tor}}$ satisfies neither the ascending chain condition nor descending chain condition: 
\begin{theorem}\label{thm:upper_lct_dense}
    Fix two positive integers $n$, $r$ with $r<n$ and a positive real number $\delta\leq\frac{1}{2}$. 
    Let $m=\lfloor\frac{1}{\delta}\rfloor$. 
    Then 
    $\dlctu_{n,r}^{\on{tor}}\cap [1-m\delta,1-\frac{m}{2m-1}\delta]$ is dense in $[1-m\delta,1-\frac{m}{2m-1}\delta]$. 
    In particular, $\dlctu_{n,r}^{\on{tor}}$ satisfies neither the ascending chain condition nor the descending chain condition. 
\end{theorem}
\begin{proof}
    Let $s$ be a large integer which is coprime to $m$ and $\{e_1,\ldots,e_n\}$ be a basis of $\bZ^n$. 
    Let $N$ be the lattice generated by $e_1$, $\ldots$, $e_n$, and $(1-\frac{m}{s})e_1+\frac{1}{s}e_2$. Then all the primitive vectors in $\{a_1e_1+\cdots+a_ne_n\mid a_i\in\mathbb R, 0\leq a_i<1\}\cap N$ are $(\lceil \frac{km}{s}\rceil-\frac{km}{s})e_1+\frac{k}{s}e_2$ where $0< k< s$. 
    Note that 
    \begin{itemize}
        \item $\frac{1}{m}<\lceil \frac{km}{s}\rceil-\frac{km}{s}+\frac{k}{s}<1+\frac{m-1}{m}$ for $0< k< s$ and
        \item for any real number $q\in (\frac{1}{m},1+\frac{m-1}{m})$, 
    there exists $\tilde k$ such that $0< \tilde k< s$ and $0<|\lceil \frac{\tilde km}{s}\rceil-\frac{\tilde km}{s}+\frac{\tilde k}{s}-q|<\frac{2m-2}{s}$. 
    \end{itemize}
    
    We consider $\sigma=\operatorname{Cone}(e_1,\ldots,e_n)$, and let $W\subseteq N_\mathbb C$ be an $r$-dimensional complex vector space such that $W\cap N=\mathbb Z((\lceil \frac{\tilde km}{s}\rceil-\frac{\tilde km}{s})e_1+\frac{\tilde k}{s}e_2)$. Then one can check by Lemma~\ref{lem:delta_lc_combinatorial} that $\mathcal F=\mathcal F_W$ is a toric foliation of rank $r$ on $X=U_\sigma$ such that $(X,\mathcal F,t)$ is $\delta$-lc if and only if $t\in[0,b_s]$, where 
    \[b_s=\frac{\lceil \frac{\tilde km}{s}\rceil-\frac{\tilde km}{s}+\frac{\tilde k}{s}-\delta}{\lceil \frac{\tilde km}{s}\rceil-\frac{\tilde km}{s}+\frac{\tilde k}{s}}.\]
    We have 
    $\lim\limits_{s\to\infty,\,\on{gcd}(m,s)=1}b_s=\frac{q-\delta}{q}$. Hence, $\frac{q-\delta}{q}$ is an accumulation point of $\dlctu_{n,r}^{\on{tor}}$. As a result, every point in $(1-m\delta,1-\frac{m}{2m-1}\delta)$ is an accumulation point of $\dlctu_{n,r}^{\on{tor}}$.
\end{proof}

\begin{example}[{$\dlctl$ does not satisfy the ACC}]
    Let $N$ be the lattice generated by $(1,0), (0,1)$ and $(\frac{1}{n},\frac{1}{n})$, and let $\sigma$ be cone of the first quadrant. 
    We consider $W=\mb C(0,1)$ and $\delta=\frac{1}{2}$. 
    Then for $n> 4$,
    $(U_{\sigma,N},\cF_W,t)$ is $\frac{1}{2}$-lc if and only if $t\in[\frac{n-4}{n-2},1]$.
    Hence $\frac{1}{2}\text{-}\underline{\on{LCT}}(U_{\sigma,N},\mc F_W)=\frac{n-4}{n-2}$, and $\{\frac{n-4}{n-2}\}_{n\in\bN,n>4}$ is a strictly increasing sequence.

\end{example}

 \begin{proposition}\label{prop_fan_to_cone}
    $\dlctl_{n,r}^{\on{tor}}$ is the set of $\dlctl(X,\mc F)$ where $\mc F$ is a toric foliation of rank $r$ on an affine toric variety $X=U_\sigma$.
\end{proposition}
\begin{proof}
    Let $\cF_W$ be a toric foliation on a toric variety $X_\Sigma$. 
    Note that 
    \[\dlctl(X_\Sigma,\cF_W)=\max\limits_{\sigma\in\Sigma\ \text{is maximal}}\dlctl(U_{\sigma},\cF_W).\]
\end{proof}


\begin{definition}\label{def_delta_L}
    Let $\delta>0$, $s\in\bN$, and a non-negative integer $\ell < s$. 
    \begin{enumerate}
        \item For any $x=(x_1,\ldots,x_s)\in\mb R^s$, we define $\psi_\ell\colon \bR^s\to\bR$ by $\psi_\ell(x) = \sum_{i=1}^\ell\{x_i\}$ and $\iota_{i,j}(x)=1$ if $\psi_i(x)=\psi_j(x)$ and $\iota_{i,j}(x)=0$ otherwise where $\{x\}:= x-\lfloor x\rfloor$. 
        \item Define $\delta\text{-}V_{s,\ell}$ to be the set of $x\in(0,1)^s\cap\mb Q^s$ such that
        \begin{enumerate}
            \item  $\on{index}(x_i)\mid\on{index}(x_1,\ldots,\widehat{x_i},\ldots,x_s)$ for each $i$,
            \item  $\delta>\psi_s(x)$, and
            \item  $t_{s,\ell,\delta}(x)\psi_\ell(mx)+(1-t_{s,\ell,\delta}(x))\psi_s(mx)\geq (1-t_{s,\ell,\delta}(x)+\iota_{s,\ell}(mx)t_{s,\ell,\delta}(x))\delta$ for all $m\in\mb Z$, where 
            \[t_{s,\ell,\delta}(x):=\frac{\delta-\psi_s(x)}{\delta-(\psi_s(x)-\psi_\ell(x))}.\]
        \end{enumerate}
        \item $\delta\text{-}L_{s,\ell}:=\{t_{s,\ell,\delta}(x)\mid x\in \delta\text{-}V_{s,\ell}\}$.
    \end{enumerate}
\end{definition}

\begin{remark}
Choose a basis $\{e_1,\ldots,e_s\}$ of the lattice $\bZ^s$, and let $N$ be the lattice generated by $e_1,\ldots,e_s$ and $x=x_1e_1+\ldots+x_se_s$. 
Let $\sigma=\on{Cone}(e_1,\ldots,e_s)$ and $W=\mb Ce_1+\cdots+\mb Ce_\ell$. 
Then the item (2a) says that each $e_i$ is primitive in $N$; the items (2b) and (2c) guarantee that $(U_{\sigma,N},\mc F_W)$ is $(t_{s,\ell,\delta}(x),\delta)$-adjoint log canonical and $\dlctl(U_{\sigma,N},\mc F_W)=t_{s,\ell,\delta}(x)$.
\end{remark}

\begin{theorem}[{\cite[Theorem 4.3]{ambro06} or \cite{lawrence4finite}}]\label{thm_subgroups_in_real_torus}
    Let $R=\mb R^d/\mb Z^d$ be a real torus.
    \begin{enumerate}
        \item Let $U\subseteq R$ be a non-empty open subset. 
        Then the set of closed subgroups of $R$ which do not intersect $U$ has finitely many maximal elements with respect to inclusion.
        \item The set of finite unions of closed subgroups of $R$ satisfies the descending chain condition.
    \end{enumerate}
\end{theorem}

\begin{lemma}\hspace{1em}
  $\delta\text{-}L_{s,\ell}$ satisfies the descending chain condition. 
\end{lemma}
\begin{proof}
    Suppose on the contrary that there exists a strictly decreasing sequence $\{t_m\}_{m\in\mb N}$ in $\delta\text{-}L_{s,\ell}$. 
    By the definition of $\delta\text{-}L_{s,\ell}$, we can write  
    \begin{itemize}
        \item $t_m=t_{s,\ell,\delta}(x^{(m)})$ for some $x^{(m)}=(x^{(m)}_1,\ldots,x^{(m)}_s)\in (0,1)^s\cap\mb Q^s$ with $\delta>\psi_s(x^{(m)})$ and \[\on{index}(x^{(m)}_i)\mid\on{index}(x^{(m)}_1,\ldots,\widehat{x^{(m)}_i},\ldots,x^{(m)}_s)\] 
        for each $i$.
        \item If $N_m\subset \mb Q^s$ is the lattice generated by $\mb Z^s$ and $x^{(m)}$, then for all $y\in N_m$ with $\iota_{s,\ell}(y)=0$ and $\delta > \psi_s(y)$, we have $t_{s,\ell,\delta}(x^{(m)})\geq t_{s,\ell,\delta}(y)$ by Definition~\ref{def_delta_L}(2c). 
    \end{itemize}
    
    Let $R=\mb R^s/\mb Z^s$ and $U_m\subset R$ be the non-empty open subset defined by
    \[
    U_m=\{y=(y_1,\ldots,y_s)\in\mb R^s/\mb Z^s\mid \iota_{s,\ell}(y)=0,\ \delta> \psi_s(y),\ \text {and}\ t_{s,\ell,\delta}(y)>t_{s,\ell,\delta}(x^{(m)})\}.
    \]
    Then we have $N_m\cap U_m=\emptyset$.
    Let $X_m$ be the union of closed subgroups of $R$ which do not intersect $U_m$. Then
    $\{X_m\}_{m\in\bN}$ is a decreasing sequence of unions of finitely many closed subgroups of $R$ by Theorem~\ref{thm_subgroups_in_real_torus}(1) and the fact that $\{U_m\}$ is increasing. 
    We have $x^{(m)}\in X_m$ as $x^{(m)}$ is in the closed subgroup $N_m$ of $R$ which does not intersect $U_m$. On the other hand, we have 
    $\iota_{s,\ell}(x^{(m)})=0$, $\delta>\psi_s(x^{(m)})$, and  $t_{s,\ell,\delta}(x^{(m)})=t_m>t_{m+1}=t_{s,\ell,\delta}(x^{(m+1)})$, 
    which implies that $x^{(m)}\in U_{m+1}$ and hence $x^{(m)}\not\in X_{m+1}$. That is, $\{X_m\}$ is strictly decreasing, contradicting Theorem~\ref{thm_subgroups_in_real_torus}(2).
\end{proof}

\begin{theorem}
    For $\delta\in(0,1]$, we have 
    $\dlctl_{n,r}^{\on{tor}}=\{0\}\cup\bigcup_{1\leq s\leq n,\,0\leq\ell\leq \min\{r,s\}}\delta\text{-}L_{s,\ell}$.
\end{theorem}
\begin{proof}
    ($\supseteq$) Note that when $N=\bigoplus_{i=1}^n\bZ e_i$, $\sigma=\on{Cone}(e_1,\ldots, e_n)$, and $W=\sum_{i=1}^r\bC e_i$, we have $\dlctl(U_{\sigma,N},\cF_W)=0$
    So $0\in\dlctl_{n,r}^{\on{tor}}$.
    
    Fix $s$, $\ell$ such that $1\leq s\leq n$ and $0\leq\ell\leq \min\{r,s\}$.
    A number is in $\delta\text{-}L_{s,\ell}$ if and only if it is of the form $t_{s,\ell,\delta}(x)$ where $x=(x_1,\ldots,x_s)\in (0,1)^s\cap\mb Q^s$ satisfies the requirements in Definition~\ref{def_delta_L}(2).
    Let $N=\mb Ze_1+\cdots+\mb Ze_n$ be a lattice, $\sigma=\on{Cone}(e_1,\ldots,e_s)$, and $N'$ be the lattice generated by $N$ and $x_1e_1+\cdots+x_se_s$. 
    Let $W\subseteq N_\bC$ be an $r$-dimensional vector subspace such that $W\cap N=\bZ e_1+\cdots+\bZ e_\ell$. 
    Then one can check that $(U_{\sigma,N'},\cF_W,t_{s,\ell,\delta}(x))$ is $\delta$-lc and $\dlctl(U_{\sigma,N'},\cF_W)=t_{s,\ell,\delta}(x)$.

    ($\subseteq$) Let $N=\mb Ze_1+\cdots+\mb Ze_n$ be a lattice, $\sigma=\on{Cone}(v_1,\ldots,v_m)\subseteq N_\mb R$ be a rational simplicial cone of dimension $m$ with each $v_i$ being primitive, and $W\subseteq N_\mb C$ be vector subspace of dimension $r$. 
    Suppose that $(X:=U_{\sigma},\cF:=\cF_W,t)$ is $\delta$-lc for some $t\in[0,1]$. We have 
    \[t\phi_{K_{\cF}}(v)+(1-t)\phi_{K_X}(v)\geq(1-t+\iota(v)t)\delta\] 
    for all primitive vectors $v\in\sigma\cap N$, where $\iota(v):=1$ if $v\in W$ and $\iota(v):=0$ otherwise. Then
    \[(\phi_{K_{\cF}}(v)-\phi_{K_X}(v)+\delta-\iota(v)\delta)t\geq \delta-\phi_{K_X}(v).\]
    Note that if $\iota(v)=1$, then we have $(\phi_{K_{\cF}}(v)-\phi_{K_X}(v)+\delta-\iota(v)\delta)\leq 0$. 
    Let \[S=\{v\in\sigma\cap N\mid v\textnormal{ is primitive, } \iota(v)=0, \textnormal{ and } \delta>\phi_{K_X}(v)\}.\]
    Note that $S$ is a finite set. 
    Let $L=\max\{0,\frac{\delta-\phi_{K_X}(v)}{\delta-(\phi_{K_X}(v)-\phi_{K_{\cF}}(v))}\mid v\in S\}$. 
     
    \begin{claim}
        $\dlctl(U_\sigma,\cF_W)=L$.
    \end{claim}
    \begin{claimproof}
        If $\dlctl(U_\sigma,\cF_W)=0$, then $\phi_{K_X}(v)\geq \delta$ for all primitive $v\in\sigma\cap N$. 
        Thus for all $v\in S$, we have 
        \[\frac{\delta-\phi_{K_X}(v)}{\delta-(\phi_{K_X}(v)-\phi_{K_{\cF}}(v))}\leq 0\] 
        and hence $L=0$. 

        If $t=\dlctl(U_\sigma,\cF_W)>0$, then for any $\varepsilon>0$, there is a primitive vector $v\in\sigma\cap N$ such that $(\phi_{K_{\cF}}(v)-\phi_{K_X}(v)+\delta-\iota(v)\delta)t\geq \delta-\phi_{K_X}(v)$ but $(\phi_{K_{\cF}}(v)-\phi_{K_X}(v)+\delta-\iota(v)\delta)(t-\varepsilon) < \delta-\phi_{K_X}(v)$. 
        Thus $\phi_{K_{\cF}}(v)-\phi_{K_X}(v)+\delta-\iota(v)\delta>0$ and hence $v\in S$. 
        So $L\geq \frac{\delta-\phi_{K_X}(v)}{\delta-(\phi_{K_X}(v)-\phi_{K_{\cF}}(v))}>t-\varepsilon$ for all $\varepsilon>0$. 
        Then $L\geq t$. 
        It is clear that $L\leq t$. 
        This completes the proof of the claim. $\quad\blacksquare$
    \end{claimproof}
    
    If $L=0$, then we are done. 
    Suppose that 
    \[L=\frac{\delta-\phi_{K_X}(\tilde v)}{\delta-(\phi_{K_X}(\tilde v)-\phi_{K_{\mc F_W}}(\tilde v))}>0\] 
    for some $\tilde{v}\in S$. 
    After re-indexing, we may assume that $\tilde v\in\on{Relint}(\on{Cone}(v_1,\ldots,v_s))$ and 
    \[\{v_1,\ldots,v_s\}\cap W=\{v_1,\ldots,v_\ell\}\] 
    for some $s$, $\ell$ with $1\leq s\leq m$ and $0\leq\ell\leq \min\{r,s\}$.
    Write $\tilde v=x_1v_1+\cdots+x_sv_s$ and let $x=(x_1,\ldots,x_s)\in\bQ^s_{>0}$. Now we claim that $x\in (0,1)^s\cap\mb Q^s$. 
    Assuming the claim, it is easy to see that $L\in \delta\text{-}V_{s,\ell}$. 
    \begin{claim} $x\in (0,1)^s\cap\mb Q^s$. \end{claim}
    \begin{claimproof}
        Suppose $x_k\geq 1$ for some $k\in\{1,\ldots,s\}$. 
        We have three cases:
        \begin{enumerate}
            \item If $k\leq \ell$, then $v'=\tilde v-v_k\in\sigma\cap N$ is not contained in $W$ as $\iota(\tilde{v})=0$. 
            Thus \[\phi_{K_X}(v')-\phi_{K_{\cF}}(v')=\phi_{K_X}(\tilde v)-\phi_{K_{\cF}}(\tilde v)<\delta\] 
            and therefore 
            \[\frac{\delta-\phi_{K_X}(v')}{\delta-(\phi_{K_X}(v')-\phi_{K_{\cF}}(v'))}>\frac{\delta-\phi_{K_X}(\tilde v)}{\delta-(\phi_{K_X}(\tilde v)-\phi_{K_{\cF}}(\tilde v))},\]
            contradicting the maximality of $L$. 

            \item If $k\geq\ell+1$ and $v'=\tilde v-v_k\in W$, then by maximality of $L$, we have 
            \begin{align*}
                & {} (\phi_{K_{\cF}}(v')-\phi_{K_X}(v')+\delta-\iota(v')\delta)L\\ 
                &=(\phi_{K_{\cF}}(\tilde v)-\phi_{K_X}(\tilde v)+1)\frac{\delta-\phi_{K_X}(\tilde v)}{\delta-(\phi_{K_X}(\tilde v)-\phi_{K_{\cF}}(\tilde v))} \\
                &\geq \delta-\phi_{K_X}(\tilde v)+1.  
            \end{align*}
            Hence $0\geq \delta^2-\delta\phi_{K_X}(\tilde v)+\phi_\cF(\tilde v)=\delta(\delta-\phi_{K_X}(\tilde v))+\phi_{\cF}(\tilde v)>0$, which is impossible. 

            \item If $k\geq\ell+1$ and $v'=\tilde v-v_k\notin W$, then 
            the similar computation shows that $0\geq \phi_{\cF}(\tilde v)$, which implies that $\ell=0$. 
            Thus, $L=1\in\delta\textnormal{-}L_{s,0}$. $\quad\blacksquare$
        \end{enumerate}
     \end{claimproof}
\end{proof}

\begin{theorem}\label{thm:lower_delta_lct_dcc}
    $\dlctl_{n,r}^{\on{tor}}$ satisfies the descending chain condition.
\end{theorem}
\begin{proof}
    It is clear since $\dlctl_{n,r}^{\on{tor}}$ is a union of finitely many sets, each of which satisfies the descending chain condition.
\end{proof}

\appendix
\section{Toroidal foliations}
In this section, we recall the definition of toroidal foliations, which is first given in \cite{CC}. 
\begin{definition}[{Toroidal foliated pairs, \cite[Definition 4.4]{CC}}]\label{defn_toroidal}
    \begin{enumerate}
        \item A foliation $\cF$ on a normal variety $X$ is \emph{toroidal} if it is formally locally toric adapted to $\Xi$ for some toroidal embedding $(X\setminus\Xi)\hookrightarrow X$. 
        That is, there exists a reduced divisor $\Xi$ on $X$ such that for any closed point $x\in X$, there exist 
        \begin{itemize}
            \item a lattice $N$, 
            \item a rational, strongly convex, polyhedral cone $\sigma\subseteq N_\bR$, 
            \item a closed point $p\in U_\sigma$, 
            \item a complex vector subspace $W_p\subseteq N_\bC$, and 
            \item an isomorphism of complete local algebras $\psi_x\colon \widehat{\cO}_{X,\,x}\cong\widehat{\cO}_{U_\sigma,\,p}$, whose induced morphism $\diff\psi_x$ on the tangent sheaves maps $\cF$ to $\cF_W$. 
        \end{itemize} 
    The divisor $\Xi$ is called the \emph{associated reduced divisor} for the toroidal foliation $\cF$. 

    We say $(U_\sigma,p,W_p)$ is a \emph{semi-local model} of $\mathcal F$ at $x$ if $\sigma$ is a top-dimensional cone, $p$ is the torus-invariant point, and $\psi_x$ maps the ideal of $\Xi$ to the ideal of some torus-invariant divisor in $U_\sigma$. 

    We say $(U_\sigma,p,W_p)$ is a \emph{local model} of $\mathcal F$ at $x$ if $p$ is in the orbit $\cO_\sigma$ and $\psi_x$ maps the ideal of $\Xi$ to the ideal of $U_\sigma\setminus T_\sigma$ where $T_\sigma$ is the maximal torus in $U_\sigma$. 

    A \emph{stratum} of $\Xi$ is a closed subvariety $Z$ of $\Xi$ such that, 
    near any $z\in Z$, $Z$ is formally locally a stratum of $U_\sigma\setminus T_\sigma$ for some local model $(U_\sigma,p,W_p)$ of $\mathcal F$ at $z$. A \emph{stratum} of $\on{Sing}(\mc F)$ is defined in the same way.

    \item
    We say a foliated pair $(\cF,\Delta)$ on a normal variety $X$ is \emph{toroidal} if $\mc F$ is toroidal with associated reduced divisor $\Xi$ and $\on{Supp}(\Delta)\subseteq \Xi$. 
    Let $(U_\sigma,p,W_p)$ be a local model of $\mc F$ at $x\in X$. Then there exists a unique torus invariant divisor $\Delta_p = \sum_{\rho\in\sigma(1)} a_\rho D_\rho$ such that in the formal neighborhood of $x\in X$, $\Delta$ is given by $\Delta_p$ via the isomorphism $\psi_x$. 
    The tuple $(U_\sigma,p,W_p,\Delta_p)$ is called a \emph{local model} of $(\cF,\Delta)$ at $x\in X$. 
    A semi-local model of $(\cF,\Delta)$ at $x\in X$ is defined in the similar way as in (1). 
    \end{enumerate}
\end{definition}

\subsection{\texorpdfstring{$\delta$-log canonical}{delta-log canonical}}
In this subsection, we show $\delta$-log canonicity of toroidal adjoint foliated structures can be check on its local models. 
\begin{lemma}\label{lem:toroidal_delta_lc_combinatorial}
    Let $(X,\cF,\Delta,t)$ be a toroidal foliated quadruple. 
    Then $(X,\cF,\Delta,t)$ is $\delta$-lc if and only if for any point $x\in X$ and any local model $(U_\sigma,p,W_p,\Delta_p)$ at $x\in X$, we have 
    \[\phi(v) \geq\begin{cases}
        (1-t)\delta & \mbox{if } v\notin W \\
        \delta & \mbox{if } v\in W
    \end{cases}\]
    for all primitive elements $v\in\sigma$ where $\phi$ is the support function of $tK_{\cF_{W_p}}+(1-t)K_{U_\sigma}+\Delta_p$. 
\end{lemma}
\begin{proof}
    Suppose $(X,\cF,\Delta,t)$ is $\delta$-lc. 
    Let $(U_\sigma,p,W_p,\Delta_p)$ be any local model at $x\in X$. 
    For any primitive element $v$ in $\sigma$, there is a sequence of blow-up along some subvarieties dominating $S_\sigma$, which is the strata of $\Xi$ and is formally locally $V_\sigma$. 
    such that $\phi(v)$ is the adjoint log discrepancy of the exceptional divisor $E$. 

    Conversely, suppose the support function $\phi_t$ has the desired inequalities as above. 
    We first take a (toroidal) resolution of singularities $\pi\colon (X',\cF',\Delta',t) \to (X,\cF,\Delta,t)$. 
    For any exceptional divisor $E$ over $X$, we would like to show its adjoint log discrepancy 
    \begin{align}\label{ineq:t_delta_lc}
        a(E,\cF,\Delta,t) &\geq
        \begin{cases}
        (1-t)\delta & \mbox{if $E$ is foliation-invariant}\\
        \delta & \mbox{otherwise}.
    \end{cases}\end{align}
    If the center $Z'$ of $E$ on $X'$ is a divisor, then $a(E,\cF,\Delta,t) = \phi(v)$ for some local model $(U_\sigma,p,W_p,\Delta_p)$ at a general point of $c_X(E)$. 
    Otherwise, $E$ is exceptional over $X'$ and by Zariski lemma, we may assume $E$ is obtained from a sequence of blow-ups starting from $X'$. 
    Precisely, there are birational morphisms 
    \[X_k\xrightarrow{\pi_{k-1}}X_{k-1}\xrightarrow{\pi_{k-2}}\ldots\xrightarrow{\pi_1}X_1:= X'\] 
    where $\pi_i$ is the blow-up along $Z_i\subseteq X_i$ for $i\in [1,k-1]\cap\bN$, $Z_1=Z'$ and each $Z_i$ dominates $Z_1$. 
    Let $E_{i}=\on{Exc}(\pi_{i-1})\subseteq X_{i}$ for $i\in [2,k]\cap\bN$. 
    Note that $Z_i\subseteq E_i$ for $i\in [2,k-1]\cap\bN$. 

    \begin{claim}
        $a(E_i,\cF,\Delta,t)\geq (t\iota(E_i) + (1-t))\delta$ for any $i\in [2,k]\cap\bN$.  
    \end{claim}
    \begin{claimproof}
        It is clear when $\cF = \cT_X$ and $\cF=0$. 
        So we assume that $\cF$ is neither $\cT_X$ nor $0$. 
        We will proceed by induction on $i$. 
        For the base case $i=2$, if $Z_1$ is not contained in any stratum of $\Xi$, then $Z$ meets the smooth locus of $\cF$ and thus, $a(E_2,\cF,\Delta,t)\geq t(1)+(1-t)(2)+t\iota(E_2)+(1-t)\geq (t\iota(E_2)+(1-t))\delta$. 
        So we may assume $Z_1$ is contained in some strata of $\Xi = \sum_i\Xi_i$, say $Z_1\subseteq \Xi_1$.  
        By \cite[Lemma 4.15]{CC}, there is a semi-local model $(U_{\sigma_1},p_1,W_1,\Delta_1)$ at a general point of $Z_1$ such that $Z_1$ and $\Xi_1$ corresponds to $V_{\tau_1}$ and $D_{\rho_1}$ respectively where $\rho_1\preceq\tau_1\preceq\sigma_1$. 
        Then we have 
        \begin{align*}
            a(E_2,\cF,\Delta,t) &= a(E_2,\cF_1,\Delta_1,t) \\
            &= ta(E_2,\cF)+(1-t)a(E_2,X)-\on{mult}_{E_2}\pi_{1}^*\Delta \\
            &= \phi_1(\sum_{\rho\in\tau_1(1)}v_\rho) \\
            &\geq\phi_1(v_{\rho_1}) \\
            &\geq (t\iota(D_{\rho_1})+(1-t))\delta \\
            &\geq (1-t)\delta
        \end{align*}
        where the second equality holds as $X_1=X'$ is smooth 
        and $\phi_1$ is the support function of $K_{(X_1,\cF_1,\Delta_1,t)}$. 
        Thus, we obtain the inequality when $\iota(E_2)=0$. 
        If $\iota(E_2)=1$, then $\sum_{\rho\in\sigma_1(1)}v_{\rho}\in W$. 
        By the non-dicriticality of $\cF$, we have $\sigma_1\subseteq W$, in particular, $\rho_1\subseteq W$. 

        Now we assume that $a(E_i,\cF,\Delta,t)\geq (t\iota(E_i)+(1-t))\delta$ for $i\in [2,j]\cap\bN$ for some $j\in\bN_{\geq 2}$. 
        At a general point on $Z_{j}$, there is a semi-local model $(U_{\sigma_{j}},p_{j},W_{j},\Delta_{j})$ such that $Z_{j}$ and $E_{j}$ corresponds to $V_{\tau_{j}}$ and $\rho_{j}$ where $\rho_{j}\preceq\tau_{j}\preceq\sigma_{j}$. 
        Let $\phi_j\colon \sigma_j\to\bR$ be the support function of $K_{(U_{\sigma_j},p_j,W_j,\Delta_j)}$. 
        \begin{align*}
        a(E_{j+1},\cF,\Delta,t) &= a(E_{j+1},\cF_j,\Delta_j,t) \\
        &= ta(E_{j+1},\cF_{j})+(1-t)a(E_{j+1},X_{j})-\on{mult}_{E_{j+1}}\pi_{j}^*\Delta_{j} \\
        &= \phi_{j}(\sum_{\rho\in\tau_{j}(1)}v_\rho) \\
        &\geq \phi_j(v_{\rho_j}) \\
        &\geq (t\iota(D_{\rho_j})+(1-t))\delta \\
        &\geq (1-t)\delta
        \end{align*}
        Thus, we obtain the inequality when $\iota(E_{j+1})=0$. 
        If $\iota(E_{j+1})=1$, then $\sum_{\rho\in\sigma_j(1)}v_{\rho}\in W$. 
        By the non-dicriticality of $\cF_j$, we have $\sigma_j\subseteq W$, in particular, $\rho_j\subseteq W$. 
    \end{claimproof}
\end{proof}

\subsection{Dicritical locus}
In this subsection, we prove Lemma~\ref{lem:dicrit_locus}. 
\begin{proof}[{Proof of Lemma~\ref{lem:dicrit_locus}}]
    Note that We may assume $X_\Sigma$ is smooth. 
    If $\on{rank}\cF=n$, then $\cF=\cT_X$ is non-dicritical and $(\tau,W)$ is not dicritical for any $\tau\in\Sigma$. 

    Suppose $\on{rank}\cF=n-1$. 
    \begin{itemize}
        \item ($\supseteq$) Let $\tau\in\Sigma$ be a cone with $(\tau,W)$ being dicritical and $x\in\cO_\tau$. 
        We put $v_0\in\on{Relint}(\tau)\cap W\cap N$. 
        By \cite[Lemma 4.14]{CC}, there is a semi-local model $(U_\sigma,p,W_p)$ of $(X,\cF)$ at $x$ such that $W_p=W_p\cap\bC\tau+\sum_{j:e_j\in W}\bC e_j$. 
        As $(\tau,W)$ is dicritical, we have $e_j\in W$ for all $\bR_{\geq 0}e_j\npreceq \tau$. 
        Thus, $v'_0:=v_0+\sum_{j:e_j\in W, \bR_{\geq 0}e_j\npreceq\tau}e_j\in\on{Relint}(\sigma)\cap W_p\cap N$. 
        We do a weighted blow-up at $p$ with respect to the star subdivision of $\sigma$ along $\bR_{\geq 0}v'_0$. 
        This will introduce a non-foliation-invariant exceptional divisor whose center is $x$. 
        Thus, $x\in\on{dicrit}(\cF_W)$. 
        \item ($\subseteq$) Let $\tau\in\Sigma$ be a cone with $(\tau,W)$ being non-dicritical and $x\in\cO_\tau$.  
        Then any exceptional divisor over $x$ is foliation-invariant by the moreover part of \cite[Lemma 4.15]{CC}. 
        Thus, $x\notin\on{dicrit}(\cF_W)$. 
    \end{itemize}
    
    Suppose $\on{rank}\cF\leq n-2$. 
    \begin{itemize}
        \item ($\supseteq$) Let $\tau\in\Sigma$ be a cone with $(\tau,W)$ being dicritical. 
        Let $v_0\in\on{Relint}(\tau)\cap W\cap N$. 
        If $\dim\tau\geq r+1$, then the star subdivision of $\Sigma$ along $\bR_{\geq 0}v_0$ introduces a non-foliation-invariant exceptional divisor whose center has dimension $n-\dim\tau\leq n-r-1$. 
        Thus, $V_\tau\subseteq\on{dicrit}(\cF_W)$. 

        If $\dim\tau\leq r$, then we take a general subvariety $Z$ of dimension $n-r-1$ on $V_\tau$. 
        By \cite[Lemma 4.15]{CC}, there is a semi-local model $(U_\sigma,p,W_p)$ of $(X,\cF)$ at a general point $x\in Z$ such that $Z$ is formally locally $V_{\tau'}$ where $\tau\preceq\tau'\preceq\sigma$ and $\rho\subseteq W$ for any $\rho\in\tau'(1)\setminus\tau(1)$. 
        Then blowing up along $Z$ will introduce an exceptional divisor which is not foliation-invariant and whose center has dimension $n-r-1$. 
        Thus, $Z\subseteq\on{dicrit}(\cF_W)$ and hence $V_\tau\subseteq\on{dicrit}(\cF_W)$. 

        \item ($\subseteq$) Let $E$ be an exceptional divisor whose center $Z$ has dimension $\leq n-r-1$. 
        Let $\tau\in \Sigma$ be a maximal one such that $Z\subseteq V_\tau$.  
        If $(\tau,W)$ is non-dicritical, then by the moreover part of \cite[Lemma 4.15]{CC}, we have $E$ is foliation-invariant. 
        Thus, $Z\nsubseteq\on{dicrit}(\cF_W)$ and therefore $V_\tau\nsubseteq\on{dicrit}(\cF_W)$. 
    \end{itemize}
\end{proof}

\section{Dicritical locus of Fano foliation of rank one}
In this section, we show the irreducibility of the dicritical loci and singular loci for Fano foliations, which are not necessarily toric, of rank one on normal projective varieties. 

\begin{proposition}\label{prop:Fano_rank_one_general}
    Let $\cF$ be a Fano foliation of rank one on a normal projective variety $X$. 
    Then $\operatorname{dicrit}(\cF) = \operatorname{Sing}(\cF)$ which is irreducible and thus connected. 
\end{proposition}
\begin{proof}
    Since $K_\cF$ is not pseudo-effective, by \cite[Theorem 3.1]{LLM23} (see also \cite[Theorem 1.1]{CP19}), there exists a non-zero algebraically integrable sub-foliation of $\cF$. 
    Since $\on{rank}\cF=1$, $\cF$ itself is algebraically integrable. 
    
    Let $\pi\colon Y\to X$ be a Property $(*)$-modification (see \cite[Definition 3.8 and Theorem 3.10]{ambro2021positivity}) such that $\cG:=\pi^*\cF$ is induced by an equi-dimensional morphism $f\colon Y\to Z$. 
    Moreover, we write $\pi^*K_\cF = K_\cG+E+G$ where $E=\sum\iota(E_i)E_i$ and $G$ is effective. 

    Note that the general fiber $F$ of $f$ is $\bP^1$ and 
    \[0 > K_\cF\cdot \pi(F) = (K_\cG+E+G)\cdot F \geq -2+\sum\iota(E_i).\]
    Thus, there is at most one $E_i$ with $\iota(E_i)=1$. 
    Moreover, by \cite[Proposition 5.8]{CJV24}, there exists a non-foliation invariant exceptional divisor and thus, there is exactly one non-foliation invariant exceptional divisor, say $E_0$. 
    So the dicritical locus of $\cF$ is $\pi(E_0)$, the center of $E_0$, which is irreducible. 

    Note that all fibers of $f$ are irreducible as $\cF$ is Fano. 
    Also there is no foliation singularity for $\cG$. 
    Thus, $\on{Sing}\cF = \pi(E_0)$, which is irreducible. 
\end{proof}

\bibliographystyle{amsalpha}
\addcontentsline{toc}{chapter}{\bibname}
\normalem
\bibliography{Fano}

\end{document}